\documentclass[final]{amsart}
\usepackage{amsmath,amsthm,amssymb}
\usepackage{amsfonts}
\usepackage{epic, eepic}
\usepackage{graphicx}
\usepackage[top=25truemm,bottom=25truemm,left=25truemm,right=25truemm]{geometry}
\usepackage{indentfirst}
\usepackage{bm}

\newtheorem{thm}{Theorem}[section]
\newtheorem{prop}{Proposition}[section]
\newtheorem{lem}{Lemma}[section]

\theoremstyle{definition}
\newtheorem{dfn}{Definition}[section]
\theoremstyle{remark}

\newtheorem{exa}{Example}[section]

\newcommand{\rank}{\operatorname{rank}}

\numberwithin{equation}{section}

\renewcommand{\phi}{\varphi}
\renewcommand{\epsilon}{\varepsilon}

\begin{document}


\title{Parallel and dual surfaces of cuspidal edges}
\author{Keisuke Teramoto}
\address{Department of Mathematics, Graduate School of Science, 
Kobe University, Kobe 657-8501, Japan}
\email{teramoto@math.kobe-u.ac.jp}
\date{October 22, 2015}

\maketitle


\begin{abstract}
We study parallel surfaces and dual surfaces of cuspidal edges. 
We give concrete forms of principal curvature and principal direction for cuspidal edges. 
Moreover, we define ridge points for cuspidal edges by using those. 
We clarify relations between singularities of parallel and dual 
surfaces and differential geometric properties of initial cuspidal edges.
\end{abstract}

\section{Introduction}
It is well-known that 
cuspidal edges and swallowtails are generic singularities of wave fronts 
in $\bm{R}^3$ (for example, see \cite{AGV}). 
There are many studies of wave fronts 
from the differential geometric viewpoint (\cite{FSUY,KRSUY,MS,MSUY,S,SUY}). 
In particular, various geometric invariants of cuspidal edges were studied by Martins and Saji \cite{MS}. 
To investigate geometric invariants of cuspidal edges, they introduced the normal form of cuspidal edges. 
On the other hand, parallel surfaces of a regular surface are fronts and might have 
singularities.
Porteous, Fukui and Hasegawa studied the singularities of parallel surfaces and caustics 
from the viewpoint of singularity theory (cf. \cite{FH,P1,P2}) when the initial surface is regular.
Porteous \cite{P1,P2} introduced the notion of ridge point for regular surfaces 
relative to principal curvature and principal direction. 
Using this notion, 
Fukui and Hasegawa \cite{FH} showed relations between 
singularities of parallel surfaces and geometric properties of initial surfaces. 

In this paper, we deal with parallel surfaces when the initial surfaces have singularities. 
In particular, we consider parallel surfaces of cuspidal edges.
Since cuspidal edges have unit normal vector fields, we can consider parallel surfaces. 
We show relations between singularities on parallel surfaces 
and geometric properties of initial cuspidal edges (Theorem \ref{relation}). 
Ridge points play an important role in studying parallel surfaces of regular surfaces, 
and also play an important role in investigating this case. 
Generally, mean curvature is unbounded at cuspidal edges. 
Thus principal curvatures might be unbounded. 
We give a condition for one principal curvature to be well-defined (in particular, finite) 
as a $C^\infty$-function at cuspidal edges (Proposition \ref{converge}).
A notion of ridge points for cuspidal edges 
is defined in Section 2 using principal curvature and principal direction. 

In Section 3, we study parallel surfaces of a cuspidal edge from the viewpoint of differential geometry. 
Moreover, we study the extended distance squared functions on cuspidal edges. 
In the case of cuspidal edges, the extended distance squared function has $D_4$ singularities or worse, 
unlike the case of regular surfaces. 
We give the conditions for distance squared functions to have $D_4$ singularities (Theorem \ref{umbilic}). 

In Section 4, we study dual surfaces and the extended height functions. 
In the case of cuspidal edges, extended height functions have $A_2$ singularities or worse. 
We define dual surfaces as a part of the discriminant set of extended height functions. 
We give relations between singularities of dual surfaces and geometric 
properties of cuspidal edges (Proposition \ref{singdual}) and show conditions 
for the extended height function to have $D_4$ singularities. 
Moreover, we give relations between 
singularities of dual surfaces and extended height functions (Proposition \ref{relatedual}). 

All maps and functions considered here are of class $C^\infty$ unless otherwise stated.

\section{Cuspidal edges}
First we recall some properties of wave fronts and frontals. 
For details, see \cite{AGV,FSUY,KRSUY,MSUY,SUY}. 

Let $f:V\to\bm{R}^3$ be a smooth map 
and $(u,v)$ be a coordinate system on $V$, where $V\subset \bm{R}^2$ is a domain.  
We call $f$ a \textit{frontal} if there exists a unit vector field $\nu$ along $f$ such that 
$L=(f,\nu):V\to T_1\bm{R}^3$ is an isotropic map, 
where $T_1\bm{R}^3$ is the unit tangent bundle of $\bm{R}^3$ 
equipped with the canonical contact structure, 
and is identified with $\bm{R}^3\times S^2$, where $S^2$ is the unit sphere.  
If $L$ gives an immersion, $f$ is called a \textit{wave front} or a \textit{front}. 
The isotropicity of $L$ is equivalent to the orthogonality condition
$$\langle df(X_p),\nu(p)\rangle,\ \ (X_p\in T_pV,\ p\in V).$$
We call $\nu$ a \textit{unit normal vector} or the \textit{Gauss map} of $f$. 
For a frontal $f$, the function $\lambda:V\to\bm{R}$ defined as 
$\lambda(u,v)=\det (f_u,f_v,\nu)(u,v)$ 
is called the \textit{signed area density function}, where $f_u=\partial f/\partial u,f_v=\partial f/\partial v$. 
A point $p\in V$ is called a \textit{singular point} of $f$ if $f$ is not an immersion at $p$. 
Let $S(f)$ be the set of singular points of $f$. A singular point $p\in S(f)$ is called 
\textit{non-degenerate} if $d\lambda(p)\neq0$ holds. 
Let $p\in S(f)$ be a non-degenerate singular point. 
Then, by the implicit function theorem, $S(f)$ is parametrized by a regular curve 
$\gamma(t):(-\epsilon,\epsilon)\to V\ (\epsilon>0)$ with $\gamma(0)=p$. 
We call $\gamma$ a \textit{singular curve} and the direction of  $\gamma'=d\gamma/dt$ a \textit{singular direction}. 
Moreover, there exists a unique non-zero vector field $\eta(t)\in T_{\gamma(t)}V$ 
up to non-zero functional scalar multiplications 
such that $df(\eta(t))=\bm{0}$ on $S(f)$. 
This vector field $\eta(t)$ is called a \textit{null vector field}.
A non-degenerate singular point $p\in S(f)$ is said to be of the \textit{first kind} if $\eta(0)$ is transverse 
to $\gamma'(0)$. Otherwise, it is said to be of the \textit{second kind}. 

A \textit{cuspidal edge} is a map-germ $\mathcal{A}$-equivalent to 
$(u,v)\mapsto(u,v^2,v^3)$ at $\bm{0}$ 
and a \textit{swallowtail} is a map-germ $\mathcal{A}$-equivalent to 
$(u,v)\mapsto(u,3v^4+uv^2,4v^3+2uv)$ at $\bm{0}$, 
where two map-germs 
$f,g:(\bm{R}^2,\bm{0})\to(\bm{R}^3,\bm{0})$ are 
\textit{$\mathcal{A}$-equivalent} if there exist a diffeomorphisms 
$\Xi_s:(\bm{R}^2,\bm{0})\to(\bm{R}^2,\bm{0})$ on the source and 
$\Xi_t:(\bm{R}^3,\bm{0)}\to(\bm{R}^3,\bm{0)}$ on the target such that $\Xi_t\circ f=g\circ\Xi_s$ holds. 
The criteria for these singularities are known.
\begin{thm}[{\cite[Proposition 1.3]{KRSUY}}]\label{criteria}
Let $f:V\to \bm{R}^3$ be a front and $p\in V$ a non-degenerate singular point of $f$. 
\begin{itemize}
\item[$(1)$] $f$ at $p$ is $\mathcal{A}$-equivalent to a cuspidal edge if and only if 
$\eta\lambda(p)\neq0$ holds.
\item[$(2)$] $f$ at $p$ is $\mathcal{A}$-equivalent to a swallowtail if and only if 
$\eta\lambda(p)=0$ and $\eta\eta\lambda(p)\neq0$ hold.
\end{itemize}
\end{thm}
Let $p=\gamma(0)$ be a non-degenerate singular point of the first kind, and set 
\begin{equation}
\psi_{\text{ccr}}(t)=\det\left(\hat{\gamma}'(t),\nu\circ\gamma(t),d\nu_{\gamma(t)}(\eta(t))\right),\label{psi1}
\end{equation}
where $\hat{\gamma}=f\circ\gamma$ and $\hat{\gamma}'=d\hat{\gamma}/dt$. 
This function is originally defined in \cite{FSUY}. 
It is well-known that 
$f$ at $p$ is a front if and only if $\psi_{\mathrm{ccr}}(0)\neq0$. (\cite{FSUY,MSUY}) 

\subsection{Normal form of cuspidal edges}
Let $f=(f_1,f_2,f_3):V\to \bm{R}^3$ be a cuspidal edge and 
$\nu=(\nu_1,\nu_2,\nu_3)$ a unit normal vector field of $f$. 
Then, by using only coordinate transformations on the source and isometries on the target, 
we obtain the following normal form of cuspidal edges (for details, see \cite{MS}).
\begin{prop}[{\cite[Theorem 3.1]{MS}}]
Let $f:(\bm{R}^2,\bm{0})\to (\bm{R}^3,\bm{0})$ be a map-germ and $\bm{0}$ a cuspidal edge. 
Then there exist a diffeomorphism $\phi:(\bm{R}^2,\bm{0})\to (\bm{R}^2,\bm{0})$ 
and an isometry-germ$\Phi:(\bm{R}^3,\bm{0})\to (\bm{R}^3,\bm{0})$ satisfying that 
\begin{eqnarray}
\Phi\circ f \circ\phi(u,v)=\left( u,\frac{a_{20}}{2}u^2+\frac{a_{30}}{6}u^3+\frac{v^2}{2},
\frac{b_{20}}{2}u^2+\frac{b_{30}}{6}u^3+\frac{b_{12}}{2}uv^2+\frac{b_{03}}{6}v^3\right)+h(u,v) \label{normal},\\
(b_{20}\geq 0,b_{03}\neq 0)\nonumber
\end{eqnarray}
where
\begin{equation*}
h(u,v)=(0,u^4h_1(u),u^4h_2(u)+u^2v^2h_3(u)+uv^3h_4(u)+v^4h_5(u,v)),
\end{equation*}
with $h_i(u)\ (1\leq i \leq 4),\ h_5(u,v)$ smooth functions.
\end{prop}
We call this parametrization the \textit{normal form of cuspidal edges}.

\subsection{Principal curvatures and principal directions}
Let $f:V\to \bm{R}^3$ be a frontal and $p\in S(f)$ a singular point of the first kind, 
where $V\subset\bm{R}^2$ is a domain. 
A coordinate system $(U;u,v)$ centered at $p$ is called \textit{adapted} if it satisfies
\begin{enumerate}
 \item[1)] the $u$-axis is the singular curve,
 \item[2)] $\eta=\partial_v$ gives a null vector field on the $u$-axis, and
 \item[3)] there are no singular points other than the $u$-axis.
\end{enumerate}
We fix an adapted coordinate system $(U;u,v)$ centered at $p$.
In this case, there exists a map 
$\psi:U\to\bm{R}^3\setminus\{\bm{0}\}$ such that $f_v=v\psi$. 
$f_u$ and $\psi$ are linearly independent, 
and we can take a unit normal vector $\nu$ of $f$ as $\nu=f_u\times\psi/||f_u\times\psi||$. 
Thus we might regard the pair $\{f_u,\psi,\nu\}$ as a frame of $f$. 

We now set the following functions:
\begin{equation}
\begin{array}{l}
\hat{E}=\langle f_u,f_u \rangle,\ \hat{F}=\langle f_u,\psi \rangle,
\ \hat{G}=\langle \psi,\psi\rangle,\\
\hat{L}=-\langle f_u,\nu_u\rangle,\ \hat{M}=-\langle\psi,\nu_u\rangle,\ 
\hat{N}=-\langle\psi,\nu_v\rangle(=-\langle\psi,\eta\nu\rangle). \label{nu2}
\end{array}
\end{equation}
We note that $-\langle f_u,\nu_v\rangle=v\hat{M}$ holds. 
If $f$ is a normal form in \eqref{normal}, we have 
$\hat{E}=1,\hat{F}=0,\hat{G}=1,\hat{L}=b_{20},\hat{M}=b_{12},\hat{N}=b_{03}/2$ at $\bm{0}$. 
In particular,  
$\hat{E}\hat{G}-\hat{F}^2\neq0$ in a neighborhood of $\bm{0}$.  
We obtain the following.
\begin{lem}\label{wein1}
It holds that 
\begin{equation*}
\nu_u=\frac{\hat{F}\hat{M}-\hat{G}\hat{L}}{\hat{E}\hat{G}-\hat{F}^2}f_u+
\frac{\hat{F}\hat{L}-\hat{E}\hat{M}}{\hat{E}\hat{G}-\hat{F}^2}\psi,\ 
\nu_v=\frac{\hat{F}\hat{N}-v\hat{G}\hat{M}}{\hat{E}\hat{G}-\hat{F}^2}f_u+
\frac{v\hat{F}\hat{M}-\hat{E}\hat{N}}{\hat{E}\hat{G}-\hat{F}^2}\psi.
\end{equation*}
\end{lem}
Let $K$ and $H$ denote the Gaussian curvature and the mean curvature of $f$. 
In this setting, we have
\begin{equation}
K=\frac{\hat{L}\hat{N}-v\hat{M}^2}{v(\hat{E}\hat{G}-\hat{F}^2)},\ \ 
H=\frac{\hat{E}\hat{N}-2v\hat{F}\hat{M}+v\hat{G}\hat{L}}{2v(\hat{E}\hat{G}-\hat{F}^2)}.\label{k_h_first}
\end{equation}
The behavior of these functions are studied in \cite{MSUY,SUY,SUY2}. 

We define the \textit{principal curvatures} and 
\textit{principal directions} for cuspidal edges. 
We recall the case of regular surfaces.
Let  $g:U\to \bm{R}^3$ be a regular surface, where $U\subset \bm{R}^2$ 
is a domain, and let
$E,F,G,L,M,N$ be the coefficients of the first and the second fundamental forms
 of $g$. 
A principal curvature $\kappa$ is a solution of
\begin{equation*}
(EG-F^2)\kappa^2-(EN-2FM+GL)\kappa+(LN-M^2)=0.\label{princi1}
\end{equation*}
Solving this equation, we get $2\kappa_1=(A+B)/(EG-F^2)$ and $2\kappa_2=(A-B)/(EG-F^2)$,
where $A=EN-2FM+GL$ and $B=\sqrt{A^2-4(EG-F^2)(LN-M^2)}$. 
These functions $\kappa_1,\kappa_2$ are smooth on $U$. 
The principal directions $(\xi_i,\zeta_i)\neq(0,0)$ corresponding to $\kappa_i\ (i=1,2)$ satisfy
\begin{equation}
\left(
\begin{array}{cc}
L &M  \\
 M&N 
\end{array}
\right)
\left(
\begin{array}{cc}
\xi_i \\
\zeta_i
\end{array}
\right)=\kappa_i
\left(
\begin{array}{cc}
E &F  \\
F &G 
\end{array}
\right)\left(
\begin{array}{cc}
\xi_i  \\
 \zeta_i
\end{array}
\right).\label{pdir1}
\end{equation}
We can take principal directions as 
$(\xi_1,\zeta_1)=(N-\kappa_1 G,-M+\kappa_1 F)$, $(\xi_2,\zeta_2)=(M-\kappa_1 F,N-\kappa_1 G)$. 

Next we consider the case of cuspidal edges. 
Let $f:V\to \bm{R}^3$ be a front and $p\in S(f)$ a singular point of the first kind.  
Let $(U;u,v)$ be an adapted coordinate system centered at $p$.
We define two functions
\begin{equation}
\hat{\kappa}_1=
\frac{\hat{A}+\hat{B}}
{2v(\hat{E}\hat{G}-\hat{F}^2)},\ 
\hat{\kappa}_2=
\frac{\hat{A}-\hat{B}}
{2v(\hat{E}\hat{G}-\hat{F}^2)}\label{cecurv}
\end{equation}
on $U$, where 
$
\hat{A}=\hat{E}\hat{N}-2v\hat{F}\hat{M}+v\hat{G}\hat{L},\ 
\hat{B}= \sqrt{\hat{A}^2-4v(\hat{E}\hat{G}-\hat{F}^2)(\hat{L}\hat{N}-v\hat{M}^2)}.
$
These are smooth on the regular set in $U$. 
However, one of these functions might be unbounded at singular points. 
By \eqref{k_h_first} and \eqref{cecurv}, we have 
$K=\hat{\kappa}_1\hat{\kappa}_2,\ 2H=\hat{\kappa}_1+\hat{\kappa}_2$. 
Thus we may regard $\hat{\kappa}_i\ (i=1,2)$ as principal curvatures of $f$, 
and $\hat{\kappa}_i$ can be rewritten as 
\begin{equation}
\hat{\kappa}_1=
\frac{2(\hat{L}\hat{N}-v\hat{M}^2)}
{\hat{A}-\hat{B}},\ 
\hat{\kappa}_2=
\frac{2(\hat{L}\hat{N}-v\hat{M}^2)}
{\hat{A}+\hat{B}}.\label{cecurv3}
\end{equation}
Here, $\hat{A}\pm\hat{B}=\hat{E}(\hat{N}\pm|\hat{N}|)$ 
hold along the $u$-axis. 
On the other hand, the function $\psi_{\mathrm{ccr}}$ in \eqref{psi1} is  
$$\psi_{\text{ccr}}(u)=\frac{||f_u(u,0)||^2(-\langle\psi,\eta\nu\rangle(u,0))}{||f_u(u,0)\times\psi(u,0)||}.$$
This implies that $f$ is a front near $q=(u,0)$ if and only if $-\langle\psi,\eta\nu\rangle(q)(=\hat{N}(q))$ 
does not vanish.
Thus we have the following.
\begin{prop}\label{converge}
Under the above setting, if $\hat{N}(q)$ is positive $($resp. negative$)$, then 
$\hat{\kappa}_2$ $($resp. $\hat{\kappa}_1$$)$ 
in \eqref{cecurv} can be extended as a $C^\infty$-function near a singular point $q$.
\end{prop}
We note that the principal curvature maps for fronts are defined and discussed their behavior in \cite{MU}. 
For details of the principal curvature maps, see \cite{MU}.

By the construction, if $\hat{N}$ is positive (resp. negative) along the $u$-axis, 
$\hat{\kappa}_2$ (resp. $\hat{\kappa}_1$) can be regarded as the \textit{principal curvature} of the cuspidal edge.
We assume that $\hat{\kappa}_2$ is a smooth on $U$ in the following.
Let $\hat{\bm{v}}=(\xi,\zeta)$ be the principal direction corresponding to $\hat{\kappa}_2$. 
In this case, equation (\ref{pdir1}) is
\begin{equation}
\left(
\begin{array}{cc}
\hat{L} &v\hat{M}  \\
v\hat{M}&v\hat{N} 
\end{array}
\right)
\left(
\begin{array}{cc}
\xi \\ 
\zeta
\end{array}
\right)=
\hat{\kappa}_2
\left(
\begin{array}{cc}
\hat{E} &v\hat{F} \\
v\hat{F} &v^2\hat{G}
\end{array}
\right)
\left(
\begin{array}{cc}
\xi \\ 
\zeta
\end{array}
\right).\label{cedir1}
\end{equation}
We can factor out $v$ from \eqref{cedir1} and obtain 
\begin{equation}
\left(
\begin{array}{cc}
\hat{L}-\hat{\kappa}_2\hat{E} &v(\hat{M}-\hat{\kappa}_2\hat{F})  \\
v(\hat{M}-\hat{\kappa}_2\hat{F})&v(\hat{N}-v\hat{\kappa}_2\hat{G}) 
\end{array}
\right)
\left(
\begin{array}{cc}
\xi \\ 
\zeta
\end{array}
\right)=\left(
\begin{array}{cc}
0 \\ 
0
\end{array}
\right).\label{cedir2}
\end{equation}
Setting 
$\hat{\bm{v}}=(\xi,\zeta)=(\hat{N}-v\hat{\kappa}_2\hat{G},-\hat{M}+\hat{\kappa}_2\hat{F})$, 
this satisfies equation (\ref{cedir2}).
Since $\hat{N}$ is a non-zero function on the $u$-axis, $\hat{\bm{v}}$ is non-zero on $U$. 
This implies that $\hat{\bm{v}}$ can be regarded as the principal direction with respect to $\hat{\kappa_2}$. 

\subsection{Ridge points}
In this section, we introduce a notion called \textit{ridge points} for cuspidal edges. 
First we recall the definition of ridge points and sub-parabolic points for regular surfaces (\cite{FH,P1}).
Let $g:U\to\bm{R}^3$ be a regular surface which has no umbilic point, and
$\kappa_i$ the principal curvatures of $g$ and 
$\bm{v}_i$ the principal directions with respect to $\kappa_i\ (i=1,2)$. 
The point $g(p)$ is called a \textit{ridge point} relative to $\bm{v}_i$ if 
$\bm{v}_i\kappa_i(p)=0$, where $\bm{v}_i\kappa_i$ is the directional derivative of 
$\kappa_i$ by $\bm{v}_i$. Moreover, $g(p)$ is called a \textit{k-th order ridge point} 
relative to $\bm{v}_i$ if $\bm{v}_i^{(m)}\kappa_i(p)=0\ (1\leq m\leq k)$ and 
$\bm{v}_i^{(k+1)}\kappa_i(p)\neq0$, where $\bm{v}_i^{(m)}\kappa_i$ is the 
directional derivative of $\kappa_i$ by $\bm{v}_i$ applied $m$ times. 
In addition, the point $g(p)$ is called a \textit{sub-parabolic point} relative to $\bm{v}_j$ 
if $\bm{v}_j\kappa_i(p)=0\ (i\neq j)$ holds. 

Let $f:V\to \bm{R}^3$ be a cuspidal edge and 
$(U;u,v)$ be an adapted coordinate system centered at $p\in S(f)$. 
Assume that $\hat{N}$ is positive on the $u$-axis. 
Then by \eqref{cedir2}, we can take $\hat{\bm{v}}=\xi\partial_u+\zeta\partial_v$ as the principal direction 
relative to the principal curvature $\hat{\kappa}_2$, where 
$(\xi,\zeta)=(\hat{N}-v\hat{\kappa}_2\hat{G},-\hat{M}+\hat{\kappa}_2\hat{F})$. 
We define the \textit{ridge points for a cuspidal edge} as follows: 
\begin{dfn}
Under the above setting, the point 
$f(p)$ is called a \textit{ridge point for $f$} relative to $\hat{\bm{v}}$ 
if $\hat{\bm{v}}\hat{\kappa}_2(p)=0$. 
Moreover, $f(p)$ is called 
a \textit{k-th order ridge point for $f$} relative to $\hat{\bm{v}}$ if 
$\hat{\bm{v}}^{(m)}\hat{\kappa}_2(p)=0\ (1\leq m\leq k)$ and 
$\hat{\bm{v}}^{(k+1)}\hat{\kappa}_2(p)\neq0$.
\end{dfn}
\begin{lem}\label{ridge}
Let $f:U\to \bm{R}^3$ be the normal form \eqref{normal} of a cuspidal edge. 
Assume that 
$\hat{\kappa}$ is the principal curvature which extends as a $C^\infty$-function near $\bm{0}$ and 
$\hat{\bm{v}}$ is the principal direction corresponding to $\hat{\kappa}$. 
Then $\bm{0}$ is a first order ridge point if and only if
 \begin{align}
&4b_{12}^3+b_{30}b_{03}^2=0,\label{ridge1}\\
&-2b_{20}^3b_{03}^4-3b_{20}(4b_{12}^2+a_{20}b_{03}^2)^2
+24\left(b_{03}^4h_2(0)
+4b_{12}^2b_{03}^2h_3(0)
-8b_{12}^3b_{03}h_4(0)+16b_{12}^4h_5(\bm{0})\right)\neq 0.\label{ridge2}
\end{align}
\end{lem}
\begin{proof}
Without loss of generality, we assume $\hat{N}$ is positive near $\bm{0}$.
Direct computations show 
$(\hat{\kappa}_2)_u(\bm{0})=b_{30}-a_{20}b_{12},\ 
 (\hat{\kappa}_2)_v(\bm{0})=-(4b_{12}^3+a_{20}b_{03}^2)/2b_{03}$, 
 $\xi(\bm{0})=b_{03}/2$ and $\zeta(\bm{0})=-b_{12}$. 
Hence we get  
$
\hat{\bm{v}}\hat{\kappa}_2(\bm{0})=(4b_{12}^3+b_{30}b_{03}^2)/(2b_{03}),
$ 
which shows \eqref{ridge1}.
Again direct computations shows 
$\xi_u(\mathbf{0})=3h_4(0),\xi_v(\mathbf{0})=-b_{20}+8h_5(\bm{0}),
\zeta_u(\mathbf{0})=a_{20}b_{20}-4h_3(0),\zeta_v(\mathbf{0})=-3h_4(0)$. 
Moreover, we have
\begin{align*}
(\hat{\kappa}_2)_{uu}(\mathbf{0})&=-2\left(a_{20}^2b_{20}+b_{20}^3+ a_{30}b_{12}-12h_2(0)+2a_{20}h_3(0)\right),\\
(\hat{\kappa}_2)_{uv}(\mathbf{0})&=\frac{1}{2b_{03}^2}\left(-a_{30}b_{03}^3+8b_{12}(-4b_{03}h_3(0)+3b_{12}h_4(0)) + 
 2a_{20}b_{03}(4b_{20}b_{12}-3b_{03}h_4(0))\right),\\
(\hat{\kappa}_2)_{vv}(\mathbf{0})&=\frac{4}{b_{03}^2}
\left(-2b_{20}b_{12}^2-6b_{12}b_{03}h_4(0)+16b_{12}^2h_5(\bm{0})+b_{03}^2(h_3(0)-2a_{20}h_5(\bm{0}))\right).
\end{align*}
Since $\hat{\bm{v}}^{(2)}\hat{\kappa}_2=(\xi(\xi)_u+\zeta(\xi)_v)(\hat{\kappa}_2)_u
+(\xi(\zeta)_u+\zeta(\zeta)_v)(\hat{\kappa}_2)_v
 +\xi^2(\hat{\kappa}_2)_{uu}+2\xi\zeta(\hat{\kappa}_2)_{uv}+\zeta^2(\hat{\kappa}_2)_{vv}$ and 
$b_{30}=-4b_{12}^3/b_{03}^2$ hold, we have completed the proof.
\end{proof}

\section{Parallel surfaces of cuspidal edges and their singularities}
In this section, we consider 
parallel surfaces of cuspidal edges and their singularities. First we recall the parallel surfaces of 
regular surfaces and their singularities. Here we treat the swallowtail singularity. In \cite{FH}, 
how other singularities appeared in the parallel surfaces of regular surfaces was discussed. 

\subsection{Parallel surfaces of regular surfaces}
Let $g:U\to \bm{R}^3$ be a regular surface which has no umbilic point, 
$\nu_g$ the unit normal vector of $g$, and
$\kappa_i\ (i=1,2)$ principal curvatures of $g$. 
The parallel surface $g_t$ of $g$ is defined by 
$
g_t(u,v)=g(u,v)+t\nu_g(u,v),
$
where $t\in \bm{R}$ is a constant. 
We can take  $\nu_g$ as 
the unit normal vector along $g_t$. 
Thus $g_t$ is a front. 
We set  
$
\lambda_{g_t}=\det ((g_t)_u,(g_t)_v,\nu_g).
$  
Using the Weingarten formula and $\kappa_i$, we can rewrite $\lambda_{g_t}$ as
\begin{equation}
\lambda_{g_t}=(1-t\kappa_1)(1-t\kappa_2)\det (g_u,g_v,\nu_g).
\end{equation}
Since $\det(g_u,\ g_v,\ \nu_g)(p)\neq0$ for any $p\in U$, 
we have $S(g_t)=\{p\in U\ |\ t=1/\kappa_i(p),\ i=1,2\}$.
The following relations are known (cf. \cite[Theorem 3.4]{FH}).
\begin{thm}[\cite{FH}]
Let $g:U\to\bm{R}^3$ be a regular surface and $g_{t}$ a parallel surface of $g$, 
where $t=1/\kappa_i(p)$. Then $f_{t_0}$ at $p$ is a swallowtail if and only if 
$p$ is a first order ridge with respect to $\bm{v}_i$ and not a sub-parabolic point relative to $\bm{v}_j$, 
where $\bm{v}_i$ is the principal direction with respect to $\kappa_i$. 
\end{thm}
\subsection{Parallel surfaces of cuspidal edges}
Let $f:V\to \bm{R}^3$ be a front, $f$ at $p\in V$ a cuspidal edge and $\nu$ a unit normal vector field of $f$.  
Take an adapted coordinate system $(U;u,v)$ centered at $p$. 
We assume that $\hat{\kappa}_2$ is smooth on $U$ (i.e. $\hat{N}>0$ on the $u$-axis) and non-zero at $p\in S(f)$. 
We define the parallel surface of $f$ as 
$
f_t(u,v)=f(u,v)+t\nu(u,v),
$ 
where $t\in \bm{R}\setminus\{0\}$ is a constant. 
We can also take $\nu$ as the unit normal vector field along $f_t$. 
The signed area density for $f_t$ is 
$
\lambda_t(u,v)=\det((f_t)_u,(f_t)_v,\nu)(u,v).
$ 
Using $\hat{\kappa}_i$, $\lambda_t$ can be rewritten as 
\begin{equation}
\lambda_t=(v-tv\hat{\kappa}_1)(1-t\hat{\kappa}_2)\det(f_u,\psi,\nu).
\end{equation}
We note that by \eqref{cecurv}, $v\hat{\kappa}_1$ is non-zero smooth function on the $u$-axis and 
$v\hat{\kappa}_1$ is proportinal to $\hat{N}+|\hat{N}|$ along the $u$-axis. 
We set $\hat{\lambda}_t=(1-t\hat{\kappa}_2)$. 
Then $S(f_t)=\{\hat{\lambda}_t=0\}$, and $p$ is a singular point of $f_t$ if and only if $t=1/\hat{\kappa}_2(p)$. 
If $f$ is a normal form (\ref{normal}) of a cuspidal edge, then $\hat{\lambda}_t(\mathbf{0})=1-b_{20}t$. 
Hence we have $t=1/b_{20}=1/\hat{\kappa}_2(\bm{0})$. 
We fix $t_0=1/\hat{\kappa}_2(p)$. 
Then $S(f_{t_0})=\{(u,v)\in U\ |\ \hat{\kappa}_2(u,v)=\hat{\kappa}_2(p)(=1/t_0)\}$. 
Moreover, we see that $p\in S(f_{t_0})$ is non-degenerate 
if and only if $d\hat{\lambda}_{t_0}(p)=-t_0d\hat{\kappa}_2(p)\neq0$. 
We consider the condition that $f_{t_0}$  at $p$ is a swallowtail. 

To apply Theorem \ref{criteria} (2), we need to determine the null vector field relative to $f_{t_0}$. 
We set $\eta_{t_0}=\ell_1(u,v)\partial_u+\ell_2(u,v)\partial_v$, 
where $\ell_i(u,v)\ (i=1,2)$ are functions on $U$.
By Lemma \ref{wein1}, 
\begin{multline}
df_{t_0}(\eta_{t_0})=\left[\ell_1\left(1+t_0\frac{\hat{F}\hat{M}-\hat{G}\hat{L}}{\hat{E}\hat{G}-\hat{F}^2}\right)+
\ell_2t_0\frac{\hat{F}\hat{N}-v\hat{G}\hat{M}}{\hat{E}\hat{G}-\hat{F}^2}\right]f_u\\
+\left[\ell_1t_0\frac{\hat{F}\hat{L}-\hat{E}\hat{M}}{\hat{E}\hat{G}-\hat{F}^2}+
\ell_2\left(v+t_0\frac{v\hat{F}\hat{M}-\hat{E}\hat{N}}{\hat{E}\hat{G}-\hat{F}^2}\right)\right]\psi \label{eta}.
\end{multline}
We set $\ell_1=\hat{N}-v\hat{\kappa}_2\hat{G},\  \ell_2=-\hat{M}+\hat{\kappa}_2\hat{F}$. 
Then $df_{t_0}(\eta_{t_0})=\bm{0}$ holds for any point $(u,v)\in S(f_{t_0})$. 
Thus we can take the null vector field  $\eta_{t_0}$ as $\bm{\hat{v}}$.  
By Theorem \ref{criteria}, we obtain the following:
\begin{thm}\label{relation}
Let $f:V\to \bm{R}^3$ be a smooth map and $f$ at $p\in V$ a cuspidal edge. 
Suppose that $\hat{\kappa}_2$ can be extended as a smooth function near $p$. 
Then the parallel surface $f_{t_0}$ of $f$, where $t_0=1/\hat{\kappa}_2(p)$, 
has a swallowtail at $p$ if and only if $d\hat{\kappa}_2(p)\neq0$ and $p$ is a first 
order ridge point of the initial surface $f$. 
\end{thm}
\begin{proof}
Take an adapted coordinate system $(U;u,v)$ centered at $p\in S(f)$. 
Then we see that $d\hat{\lambda}_{t_0}=-t_0d\hat{\kappa}_2,$ and 
$\eta_{t_0}\hat{\lambda}_{t_0}=\hat{\bm{v}}(1-t_0\hat{\kappa}_2)=-t_0\hat{\bm{v}}\hat{\kappa}_2$, 
$\eta_{t_0}^{(2)}\hat{\lambda}_{t_0}=-t_0\hat{\bm{v}}^{(2)}\hat{\kappa}_2$. 
This completes the proof.
\end{proof}
Using Lemma \ref{ridge} and Theorem \ref{relation}, we have the following lemma.
\begin{lem}\label{swallow}
Let $f:V\to \bm{R}^3$ be a normal form of cuspidal edge in $(\ref{normal})$ and $f_{t_0}$ the parallel surface of $f$,
where $t_0=1/b_{20}$. 
Then $f_{t_0}$ is a swallowtail at the origin 
if and only if the coefficients of the normal form satisfy
\begin{align}
&b_{30}-a_{20}b_{12}\neq 0\ \textit{or}\ 4b_{12}^2+a_{20}b_{03}^2\neq 0,\\ 
&4b_{12}^3+b_{30}b_{03}^2=0,\\
&-2b_{20}^3b_{03}^4-3b_{20}(4b_{12}^2+a_{20}b_{03}^2)^2
+24(b_{03}^4h_2(0)
+4b_{12}^2b_{03}^2h_3(0)
-8b_{12}^3b_{03}h_4(0)+16b_{12}^4h_5(0,0))\neq 0.
\end{align}
\end{lem}
\begin{exa}
Let $f$ be a cuspidal edge given as 
$f(u,v)=(u,u^2/2+u^3/3+v^2/2,u^2+v^3/3)$. 
The coefficients of $f$ satisfy the conditions of Lemmas 2.2 and 3.1. 
The unit normal vector of $f$ is 
$
\nu=(-2u+uv+u^2v,-v,1)/\delta,
$ 
where $\delta=\sqrt{1+v^2+(-2u+uv+u^2v)^2}$. Since $b_{20}=2$, we take the parallel surface $f_{t_0}$ as 
$f_{t_0}=f+\nu/2$. We can see a swallowtail singularity on $f_{t_0}$ (see Figure \ref{fig:one}).
\begin{figure}[h]
 \begin{center}
  \includegraphics[width=7.0cm,angle=0,clip]{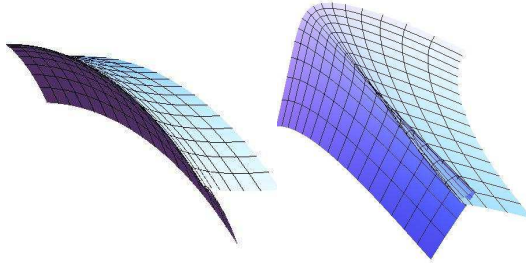}
 \caption{The left side is $f$ and the right is $f_{t_0}$ in Example 3.1.}
  \label{fig:one}
 \end{center}
\end{figure}
\end{exa}

\subsection{Extended distance squared functions on cuspidal edges}
Let $h:(\bm{R}^n,\mathbf{0})\to(\bm{R},0)$ be a function-germ. A function-germ 
$H:(\bm{R}^n\times\bm{R}^r,\mathbf{0})\to(\bm{R},0)$ is called an \textit{unfolding} of $f$ if 
$H(\bm{u},\bm{0})=h(\bm{u})$ holds. 
We define the \textit{discriminant set $\mathcal{D}_H$} of $H$ by
\begin{equation*}
\mathcal{D}_H=\left\{\bm{x}\in(\bm{R}^r,\bm{0})\ |\ 
H(\bm{u},\bm{x})=H_{u_1}(\bm{u},\bm{x})=\cdots=H_{u_n}(\bm{u},\bm{x})=0\right\},
\end{equation*} 
where 
$(\bm{u},\bm{x})=(u_1,\ldots,u_n,x_1,\ldots,x_r)\in(\bm{R}^n\times\bm{R}^r,\mathbf{0})$ 
and $H_{u_i}=\partial H/\partial u_{i}\ (1\leq i\leq n)$. 
In the case $n=1,\ r=3$, if $h'(0)=h''(0)=h'''(0)=0,\ h^{(4)}(0)\neq0$ and $H$ is $\mathcal{K}$-versal, 
then $\mathcal{D}_H$ is locally diffeomorphic to the image of swallowtail (\cite[Section 6]{BG}).
See \cite[Section 8]{AGV}, for the $\mathcal{K}$-versality. See also \cite{IT1,IT2}. 

Let $f:V\to\bm{R}^3$ be a front and $f$ at $p\in V$ a cuspidal edge. 
Take an adapted coordinate system $(U;u,v)$ centered at $p$. 
We define the following function:
\begin{equation}
\Phi:\left(U\times\bm{R}^3,(p,q)\right)\to\bm{R},\ 
\Phi(u,v,(x,y,z))=-\frac{1}{2}\left(||(x,y,z)-f(u,v)||^2-t_0^2\right),\label{Phi}
\end{equation}
where $
q=(x_0,y_0,z_0)=f(p)+t_0\nu(p)\in\bm{R}^3\setminus\{\bm{0}\},\ t_0=1/\hat{\kappa}_2(p)\in\bm{R}\setminus\{0\}
$ are constants. 
We call this function in \eqref{Phi} the 
\textit{extended distance squared function}. 
If $f$ is regular, it is known that $\Phi$ is  $\mathcal{K}$-versal and 
$\mathcal{D}_{\Phi}$ is equal to the image of parallel surfaces of $f$. 
Thus singularities of $\mathcal{D}_{\Phi}$ correspond to singularities of the parallel surface. 
Let us assume that $f$ at $p$ is cuspidal edge and $(U;u,v)$ an adapted coordinate system centered at $p$. 
Since 
$\Phi_u=-\langle (x,y,z)-f,f_u\rangle,\ 
\Phi_v=-v\langle (x,y,z)-f,\psi\rangle,$ 
we have
\begin{equation}\label{discrim1}
\mathcal{D}_{\Phi}=\{(x,y,z)\in\bm{R}^3\ |\ (x,y,z)=f(u,v)\pm t_0\nu(u,v),\ \text{for some}\ (u,v)\in U\}
\cup\{f_u\}^{\perp}.
\end{equation}
Setting $$\phi(u,v)=\Phi(u,v,q),$$ then 
$\Phi$ is an unfolding of $\phi$, but never a $\mathcal{K}$-versal unfolding by \eqref{discrim1}. 

Let $p=\bm{0}$ and $f$ be a normal form in \eqref{normal}. Then we have 
\begin{equation*}
\varphi(u,v)=\frac{1}{6b_{20}}(b_{30}u^3+3b_{12}uv^2+b_{03}v^3)+O(4),
\end{equation*}
where $O(4)=
\left\{h(u,v):(\bm{R}^2,\bm{0})\to(\bm{R},0)\ |\ (\partial^{i+j}/\partial u^i\partial v^j)h(\bm{0})=0,\ i+j\leq3\right\}$.
By a direct calcuation,  
\begin{align}\label{delta}
\Delta_{\varphi}&=\left((\varphi_{uuu})^2(\varphi_{vvv})^2-6\varphi_{uuu}\varphi_{uuv}\varphi_{vvv}
-3(\varphi_{uuv})^2(\varphi_{uvv})^2
+4(\varphi_{uuv})^3\varphi_{vvv}+4\varphi_{uuu}(\varphi_{uvv})^3\right)(\bm{0})\\
&=\frac{b_{30}}{b_{20}}(b_{30}b_{03}^2+4b_{12}^3).\notag
\end{align}
holds. 
By \cite[Lemma 3.1.]{S}, $\varphi$ at $\bm{0}$ is right equivalent to 
$(u^3+uv^2)$ $(\text{resp.}\ (u^3-uv^2))$ if and only if 
$\Delta_{\varphi}>0$ $(\text{resp.}\ \Delta_{\varphi}<0)$, 
where function-germs $h,k:(\bm{R}^2,\bm{0})\to(\bm{R},0)$ are \textit{right equivalent} 
if there exists a diffeomorphism-germ $\theta:(\bm{R}^2,\bm{0})\to(\bm{R}^2,\bm{0})$ such that $k=h\circ\theta$ holds. 
We have the following property.
\begin{thm}\label{umbilic}
Under the above settings, the function $\varphi$ has a $D_4$ singularity at $\bm{0}$ if and only if
the conditions that $b_{30}\neq0$ and $\bm{0}$ is not ridge hold.
\end{thm}
The term $b_{30}$ is called the \textit{edge inflectional curvature}. See \cite{MS} for details.
By Theorem \ref{umbilic}, we might say that $\bm{0}$ satisfying the above conditions 
is an \textit{umbilic point} of a cuspidal edge.

By \eqref{Phi}, contact between $f$ and the sphere whose center is 
$f(p)+t_0\nu(p)$ with radius $t_0=1/\hat{\kappa}_2(p)$ 
become $D_4$ type if the conditions in Theorem \ref{umbilic} hold.
\section{Dual surfaces and extended height functions}
We consider the \textit{dual surface} of a cuspidal edge and the \textit{extended height function} on a cuspidal edge.
\subsection{Dual surfaces and height functions}
Let $f:U\to\bm{R}^3$ be a front, $\nu$ a unit normal vector field of $f$ and $f$ at $p\in U$ a cuspidal edge.
Take a constant vector $\bm{c}=(c_1,c_2,c_3)\in\bm{R}^3\setminus\{\bm{0}\}$ 
which satisfies $\langle \nu(p),\bm{c}\rangle\neq0$. 
We set the map $\bar{f}:U\to\bm{R}^3$ as $\bar{f}=f+\bm{c}$, and set the function 
$\rho(u,v)=\langle \bar{f}(u,v),\nu(u,v)\rangle\ ((u,v)\in U)$. 
We define the \textit{dual surface} $f^*:U\to\bm{R}^3$ as
\begin{equation}
f^*(u,v)=\rho(u,v)\nu(u,v). \label{dual}
\end{equation}
We define the following function:
\begin{equation}
\tilde{H}:\left(U\times(S^2\times\bm{R}),(p,(\bm{n}_0,r_0))\right)\to\bm{R},\ 
\tilde{H}(u,v,(\bm{n},r))=\langle\bar{f}(u,v),\bm{n}\rangle-r\ (\bm{n}_0\in S^2, r_0\in\bm{R}).
\end{equation}
We call this function $\tilde{H}$ the \textit{extended height function}. 
It is well-known that if $f$ is regular, then $\tilde{H}$ is $\mathcal{K}$-versal and $\mathcal{D}_{\tilde{H}}$
is diffeomorphic to the image of dual surfaces (cf. \cite{IT2}). 
Hence singularities of $\mathcal{D}_{\tilde{H}}$ and the dual surface are supported. 
We set the function $\tilde{h}(u,v)=\tilde{H}(u,v,(\bm{n}_0,r_0))$. 
By definition, $\tilde{H}$ is an unfolding of $\tilde{h}$, but not $\mathcal{K}$-versal. 
Since $\tilde{H}=\langle\bar{f},\bm{n}\rangle-r,\ 
\tilde{H}_u=\langle f_u,\bm{n}\rangle$ and 
$\tilde{H}_v=v\langle\psi,\bm{n}\rangle$, we see that 
\begin{equation}
\mathcal{D}_{\tilde{H}}=
\{(\pm\nu(u,v),\pm\langle\bar{f}(u,v),\nu(u,v)\rangle)\ |\ (u,v)\in U\}\cup\{f_u\}^{\perp} \label{discrim2}.
\end{equation}
We define the map 
$
\Psi:S^2\times(\bm{R}\setminus\{0\})\to\bm{R}^3\setminus\{\mathbf{0}\},$ 
by $
\Psi(\bm{n},r)=r\bm{n}.
$ 
Setting $\bm{R}_+=\{r\in\bm{R}\  |\ r>0\}$, 
$\Psi|{S^2\times\bm{R}_+}$ is a diffeomorphism, and we have 
$
\Psi(\mathcal{D}_{\tilde{H}})\supset f^*(U).
$ 
Thus we can regard the image of $f^*$ as a part of $\mathcal{D}_{\tilde{H}}$ .

\subsection{Singularities of dual surfaces}
Let us consider singularities of dual surface.
Setting a unit normal $\nu^*$ of $f^*$ as $$\nu^*=(f^*)_u\times(f^*)_v/||(f^*)_u\times(f^*)_v||,$$ then 
$\nu^*$ is smooth on $U$. 
If $f$ is a normal form (\ref{normal}), then the signed area density $\lambda^*=\det((f^*)_u,(f^*)_v,\nu^*)$ of $f^*$ 
satisfies $\lambda^*(\mathbf{0})=b_{20}b_{03}c_3||\bm{c}||/2$. 
Thus we have the following. 
\begin{prop}
A point $p\in U$ is a singular point of $f^*$ if and only if $\hat{\kappa}_2(p)=0$. 
\end{prop}
We assume $b_{20}(=\hat{\kappa}_2(\mathbf{0}))=0$ in the following. 
We detect the null vector field $\eta^*$ of $f^*$. 
Let $\eta^*=\ell_1^*\partial_u+\ell_2^*\partial_v$ be a vector field, 
where $\ell_i^*\ (i=1,2)$ are functions on $U$. 
By Lemma 2.1,
\begin{equation}
df^*(\eta)=(\ell_1^*\rho_u+\ell_2^*\rho_v)\nu+\rho(\ell_1^*\alpha+\ell_2^*\tilde{\alpha})f_u+
\rho(\ell_1^*\beta+\ell_2^*\tilde{\beta})\psi\label{null2} ,
\end{equation}
where $\nu_u=\alpha f_u+\beta\psi,\nu_v=\tilde{\alpha}f_u+\tilde{\beta}\psi$.  
If $f$ is a normal form (\ref{normal}) and $b_{20}=0$, then $\rank df_{\bm{0}}^*=1$. Moreover, 
\begin{equation}
\rho_u(\mathbf{0})=-b_{12}c_2,\ \rho_v(\mathbf{0})=-b_{03}c_2/2,\ 
\alpha(\mathbf{0})=\tilde{\alpha}(\mathbf{0})=0,\ 
\beta(\mathbf{0})=-b_{12},\ \tilde{\beta}(\mathbf{0})=-b_{03}/2(\neq0)
\label{null4}
\end{equation}
holds. 
Setting 
$
\eta^*=\tilde{\beta}(u,v)\partial_u-\beta(u,v)\partial_v, 
$
then by (\ref{null4}), $\eta^*$ is a null vector field of $f^*$. 
Since, $\eta^*\nu^*\neq0$ holds, $f^*$ is a front. 
By Theorem \ref{criteria} (1), we consider the conditions that make dual surface $f^*$ a cuspidal edge. 
We have $d\lambda^*(\eta^*)=\tilde{\beta}(\lambda^*)_u-\beta(\lambda^*)_v$. 
If $f$ is a normal form (\ref{normal}), we obtain
\begin{equation}
(\lambda^*)_u(\mathbf{0})=\frac{1}{2}(b_{30}-a_{20}b_{12})b_{03}c_3||\bm{c}||,\ 
(\lambda^*)_v(\mathbf{0})=-\frac{1}{4}(4b_{12}^2+a_{20}b_{03}^2)c_3||\bm{c}||.
\end{equation}
Thus we have 
$
4d\lambda_{\bm{0}}^*(\eta^*)
 = -(4b_{12}^3+b_{30}b_{03}^2)c_3||\bm{c}||.
$ 
By Lemma 2.2, we have the follwing.
\begin{prop}\label{singdual}
Let $f:U\to\bm{R}^3$ be a front, $f$ at $p$ be a cuspidal edge and $f^*$ a dual surface of $f$. 
Then $f^*$ has a cuspidal edge at $p\in S(f)$ if and only if $\hat{\kappa}_2(p)=0$ 
and $p$ is not a ridge point. 
\end{prop}
\begin{exa}
Let $f$ be a cuspidal edge defined by 
$f(u,v)=(u, 2u^3/3 + v^2/2, 1+ u^3/3 + uv^2 + v^3/6)$.   
The coefficients of $f$ satisfy the conditions in Proposition 4.2. 
Thus $f^*$ has a cuspidal edge at $\bm{0}$.
The unit normal vector of $f$ is 
$\nu=(-u^2+4u^3+u^2v-v^2,-2u-v/2,1)/\delta,$ 
where $\delta=\sqrt{1 + (2 u + v/2)^2 + (u ^2- 4u^3 - u^2v+ v^2)^2}$. 
The shapes of $f$ and $f^*$ are shown in Figure \ref{fig:three}. 
\begin{figure}[h]
 \begin{center}
  \includegraphics[width=6.0cm,angle=0,clip]{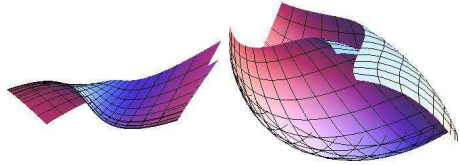}
 \caption{The left side is $f$ and the right is $f^*$ in Example 4.1.}
  \label{fig:three}
 \end{center}
\end{figure}
\end{exa}
\subsection{Singularities of extended height functions}
We consider singularities of extended height function. 
Let $f$ be a normal form in \eqref{normal}. 
We assume $b_{20}=\hat{\kappa}_2(\bm{0})=0$. 
Let $\bm{n}_0=\nu(\bm{0})=(0,0,1)$ and $r_0=\langle \bar{f}(\bm{0}),\bm{n}_0(\mathbf{0})\rangle=c_3$. 
Then a function $\tilde{h}$ becomes $\tilde{h}(u,v)=f_3(u,v)$, where $f=(f_1,f_2,f_3)$. 
Here, $\tilde{h}$ is 
\begin{equation}
\tilde{h}(u,v)=\frac{b_{30}}{6}u^3+\frac{b_{12}}{2}uv^2+\frac{b_{03}}{6}v^3+
u^4h_2(u)+u^2v^2h_3(u)+uv^3h_4(u)+v^4h_5(u,v).\label{tilh2} 
\end{equation}
Now $j^2\tilde{h}=0$ holds. 
By direct computation, 
\begin{align*}
\Delta_{\tilde{h}}&=\left((\tilde{h}_{uuu})^2(\tilde{h}_{vvv})^2-6\tilde{h}_{uuu}\tilde{h}_{uuv}\tilde{h}_{vvv}
-3(\tilde{h}_{uuv})^2(\tilde{h}_{uvv})^2
+4(\tilde{h}_{uuv})^3\tilde{h}_{vvv}+4\tilde{h}_{uuu}(\tilde{h}_{uvv})^3\right)(\bm{0})\notag\\
&=b_{30}(b_{30}b_{03}^2+4b_{12}^3).
\end{align*}
By \cite[Lemma 3.1]{S}, we have the following.
\begin{lem}\label{singheight}
Under the above conditions, the function $\tilde{h}$ has a $D_4$ singularity at $p\in S(f)$ if and only if 
$\hat{\kappa}_2(p)=0$, the edge inflectional curvature does not vanish at $p$ and $p$ is not a ridge point.
\end{lem}
By Proposition \ref{singdual} and Lemma \ref{singheight}, we have the following.
\begin{prop}\label{relatedual}
Let $f:U\to\bm{R}^3$ be a cuspidal edge and $f^*$ the dual surface of $f$. 
If the function $\tilde{h}$ has a $D_4$ singularity at $p\in S(f)$, then $f^*$ has a cuspidal edge at $p$.
\end{prop}
\subsection*{Acknowledgements}
The author would like to express his gratitude to Professors 
Kentaro Saji  and Wayne Rossman for fruitful discussions 
and helpful comments. 



\end{document}